\documentclass[12pt,letterpaper]{amsart}

\newtheorem{theorem}{Theorem}[section]
\newtheorem{lemma}[theorem]{Lemma}
\theoremstyle{definition}
\newtheorem{definition}[theorem]{Definition}

\newtheorem{corollary}[theorem]{Corollary}
\newtheorem{proposition}[theorem]{Proposition}
\theoremstyle{remark}

\numberwithin{equation}{section}

\begin{document}

\def\C{\mathbb C}
\def\R{\mathbb R}
\def\X{\mathbb X}
\def\Z{\mathbb Z}
\def\Y{\mathbb Y}
\def\Z{\mathbb Z}
\def\N{\mathbb N}
\def\cal{\mathcal}
\def\cD{\cal D}
\def\tD{\tilde{{\cal D}}}
\def\F{\cal F}
\def\tf{\tilde{f}}
\def\tg{\tilde{g}}
\def\tu{\tilde{u}}

\def\cal{\mathcal}
\def\b{\mathcal B}
\def\c{\mathcal C}
\def\cc{\mathbb C}
\def\x{\mathbb X}
\def\r{\mathbb R}
\def\uu{(U(t,s))_{t\ge s}}
\def\vv{(V(t,s))_{t\ge s}}
\def\xx{(X(t,s))_{t\ge s}}
\def\yy{(Y(t,s))_{t\ge s}}
\def\zz{(Z(t,s))_{t\ge s}}
\def\ss{(S(t))_{t\ge 0}}
\def\tt{(T(t,s))_{t\ge s}}
\def\rr{(R(t))_{t\ge 0}}

\title[Almost periodic solutions of FDE]{
Almost periodic solutions of periodic linear partial functional differential equations}

\author{Vu Trong Luong}
\address{Department of Mathematics, Tay Bac University, Son La City, Son La, Vietnam}
\email{vutrongluong@gmail.com}

\author{Nguyen Van Minh}

\address{Department of Mathematics and Statistics, University of Arkansas at Little Rock, 2801 S University Ave, Little Rock, AR 72204. USA}
\email{mvnguyen1@ualr.edu}

\thanks{The authors would like to thank the anonymous referee for his/her carefully reading the manuscript and suggestions to improve the paper's presentation.}

\date{\today}

\begin{abstract}
We study conditions for the abstract periodic linear functional differential 
equation $\dot{x}=Ax+F(t)x_t+f(t)$ to have almost periodic 
with the same structure of frequencies as $f$. The main conditions are stated in terms of 
the spectrum of the monodromy operator associated with the equation and the frequencies of the forcing term $f$. 
The obtained results extend recent results on the subject. A discussion on how the results could be extended to the case when $A$ depends on $t$ is given.
\end{abstract}
\maketitle

\section{Introduction} \label{section 1}

In this paper we consider the existence and uniqueness of almost periodic solutions with the same structure of spectrum as $f$ in equations of the following form
\begin{equation}\label{FDE}
\frac{dx(t)}{dt}=Ax(t)+F(t)x_t +f(t),\ \ x \in {\X}, t \in {\R},
\end{equation}
where the (unbounded) linear operator $A$ generates a strongly continuous semigroup and the bounded linear operator $F(t)$ is periodic
and is defined as follows, $x_t\in C_r:=C([-r,0],{\X})$, $x_t(\theta ):=x(t+\theta )$, 
$r>0$ is a given positive real number,
$F(t)\varphi :=\int^0_{-r}d\eta (t,s)\varphi (s), \ \forall \varphi \in C_r$,
$\eta (t,\cdot ) : C_r\to L({\X})$ is periodic in $t$, of bounded variation, and 
$\sup_{t} \| F(t)\| <\infty$, and $f$ is a ${\X}$-valued almost periodic function. A discussion on how the results could be extended to the case when $A$ depends on $t$ periodically will be given at the end of the paper.

\bigskip
In the theory of ordinary differential equations one of the questions that are of interest to many researchers is when exist periodic solutions to equations of the form
$$
\frac{dx}{dt}=B(t)x+f(t), t \in {\R}, x \in {\C}^n, \eqno(F)
$$
where $f$ is periodic, and $B(t)$ is a $n\times n$-matrix that is periodic with the same period as $f(t)$. A famous Massera's Theorem (\cite{mas}) says that Eq. (F) has a periodic solution with the same period as $B$ and $f$ if and only if it has a solution that is bounded on the positive half line. In addition, the periodic solution is unique if $1$ is not an eigenvalue of the monodromy operator.
Since then there have been many efforts to extend this classic result to various classes of equations and functions (see e.g. \cite{bathutrab,burzha,furnaimin,hatkri,murnaimin,murnaimin2,naimin,
naiminmiyshi,naiminshi,shinai}).  We refer the reader to some recent developments \cite{furnaimin,hinnaiminshi,naiminmiyshi,naiminshi,
shinai} and their references for more recent information in this direction. We note that the results on the existence of periodic solutions are usually proved via the existence of fixed points of the monodromy operator (or, period map)  (see e.g. \cite{burzha,lilimli,
shinai}).
Among the research methods used in this direction we note that when $f$ is almost periodic the monodromy operator method is no longer applicable because the system is no longer periodic. Instead, one uses a new method that is based on the concept of {\it evolution semigroups} associated with the evolutionary processes generated by the equations. Also, the requirement that the period of the solutions be the same as that of the forcing term $f$ will be understood as a requirement on the frequencies of the solutions that are not more than those of $f$. This justifies the introduction of the concept of {\it spectrum of a function} that allows us to measure the set of frequencies of a function on the real line. As is known, a fundamental technique of research in the ODE and FDE is variation-of-constants formulas (VCF) in the phase space. In the case of abstract functional differential equations, the VCF in the phase space is no longer valid. Instead, a weak version may make sense. In this short paper we will recall briefly  these concepts and related results in the next section. We will present an extension of the Massera's Theorem for almost periodic solutions of Eq. (\ref{FDE}) (Theorems \ref{the main} and  \ref{cor main1}). We prove that the condition of existence of bounded solutions could be removed and the equations always have a unique almost periodic solutions with frequencies as $f$ if the part of spectrum of the monodromy operator on the unit circle does not intersect the spectrum of $f$. To our best knowledge the results obtained in this paper extends some previous ones in \cite{bathutrab,furnaimin,naimin}, and complements many other results in \cite{bathutrab,miykimnaishi,murnaimin,murnaimin2,naiminmiyshi,ruevu}. In \cite{murnaimin2} the authors showed that if $A$ generates a compact $C_0$-semigroup the existence of almost periodic solutions to Eq. (\ref{FDE}) could be reduced to the finite dimensional case of ODE, so the problem could be thoroughly studied. The novelty of our results obtained in this paper is that we study the problem when $A$ generates any $C_0$-semigroup, (and even more generally, when $A$ is a family of operators that generates a periodic evolutionary process). This makes the part of spectrum on the unit circle more complicated and the nature of the problem is not of finite dimension. Finally, we give a discussion on how the obtained results could be extended to the case when $A$ may depend on time $t$ periodically. In this case without the variation-of-constants in the phase space the main results are still true though their proofs will be adjusted.

\section{Preliminaries} \label{section 2}
\subsection{Notation}
Throughout the paper we will use the following notations: 
${\N}, {\Z},
{\R}, {\C}$ stand for 
the sets of natural, integer, real, complex numbers, respectively. 
$\Gamma$ denotes the unit circle
in the complex plane ${\C}$. 
For any complex number $z$ the notation $\Re z$ stands for its real part.
${\X}$ will denote
a given complex Banach space. Given two Banach spaces 
${\X},{\Y}$ by $L({\X},{\Y})$  we will 
denote the space of all bounded linear operators from  ${\X}$ 
to  ${\Y}$. As usual,  $\sigma (T), \rho (T), R(\lambda ,T)$ 
are the notations of the spectrum, resolvent set and resolvent of the operator
 $T$.
The notations \ $BC({\R},{\X}), BUC({\R},{\X}),
AP({\X})$ \ will stand for the spaces of all ${\X}$-valued
bounded continuous, bounded uniformly continuous functions on ${\R}$ and its subspace of almost 
periodic (in Bohr's sense) functions, respectively.  

\subsection{Circular Spectrum of Functions}
Below we will introduce a transform of a function $g\in L^\infty
(\R,\X)$ on the real line that leads to a concept of spectrum of a
function. This spectrum coincides with the set of
$\overline{e^{isp(g)}}$ if in addition $g$ is  uniformly continuous, where $sp(g)$ denotes the Beurling spectrum of $g$. All results mentioned below on the circular spectrum of a function could be found in \cite{mingasste}.

\bigskip
Let $g\in L^ \infty (\R,\X)$. Consider the complex function ${\cal
S}g(\lambda )$ in $\lambda \in \C \backslash \Gamma$ defined as
\begin{equation}\label{transform}
{\cal S}g(\lambda ):= R(\lambda ,S)g,\quad \lambda \in \C \backslash
\Gamma .
\end{equation}
Since $S$ is a translation, this transform is an analytic function
in $\lambda \in \C \backslash \Gamma $.
\begin{definition}\rm
The {\it circular spectrum} of $g\in L^\infty (\R,\X)$ is defined to
be the set of all $\xi_0\in \Gamma$ such that ${\cal S}g(\lambda )$
has no analytic extension into any neighborhood of $\xi_0$ in the
complex plane. This spectrum of $g$ is denoted by $\sigma (g)$ and
will be called for short {\it the spectrum of $g$} if this does not
cause any confusion. We will denote by $\rho (g)$ the set $\Gamma
\backslash \sigma (g)$.
\end{definition}

\begin{proposition}\label{pro 2.1}
Let $\{ g_n\}_{n=1}^\infty \subset L^\infty (\R,\X)$ such that
$g_n\to g\in L^\infty (\R,\X)$, and let $\Lambda$ be a closed subset
of the unit circle. Then the following assertions hold:
\begin{enumerate}
\item \ $\sigma (g)$ is closed. \item \ If $\sigma (g_{n}) \subset
\Lambda$ for all $n \in {\N}$, then $\sigma (g)\subset \Lambda $.
\item \ $\sigma ({\cal A}g)\subset \sigma (g)$ for every bounded linear operator ${\cal A}$ acting in 
$BUC(\R,\X)$ that commutes with $S$.
\item If $\sigma (g) =\emptyset$, then $g=0$.
\end{enumerate}
\end{proposition}
\begin{proof}
For i), ii) and iv) the proofs are given in \cite{mingasste}. For iii) the proof is obvious from the definition of the circular spectrum.
\end{proof}
\begin{corollary}\label{cor 2}
Let $\Lambda$ be a closed subset of the unit circle and ${\cal F}$ be one of the function spaces $BUC(\R,\X), AP(\X)$. Then, the set
\begin{equation}\label{space Lambda (X)}
\Lambda_\F (\X) := \{ g\in {\cal F}|\ \sigma (g)\subset \Lambda \}
\end{equation}
is a closed subspace of $\F$.
\end{corollary}

\begin{lemma}\label{lem spec of S on Lambda (X)}
Let $\Lambda$ be a closed subset of the unit circle and ${\cal F}$ be one of the function spaces $BUC(\R,\X), AP(\X)$. Then, the
translation operator $S$ leaves the space $\Lambda_\F (\X)$
invariant. Moreover,
\begin{equation}\label{spec of S on Lambda (X)}
\sigma (S|_{\Lambda_\F (\X)})=\Lambda .
\end{equation}
\end{lemma}

\medskip
Below we will recall the concept of Beurling spectrum of a function. 
We denote by $F$ the Fourier transform, i.e.
\begin{equation}
({F}f)(s):= \int^{+\infty }_{-\infty }e^{-ist}f(t)dt
\end{equation}
$(s\in {\R}, f\in L^1({\R}))$. Then the {\it Beurling spectrum}\index{Beurling spectrum} 
of $u\in BUC({\R},{\X})$ is defined to be the following set
\begin{eqnarray*}
sp(u)&:=& \{ \xi \in {\R}: \forall \epsilon >0 \ \ \exists f\in L^1({\R}),\\
&& \hspace{1cm} supp{F}f \subset (\xi -\epsilon ,\xi +\epsilon ), f*u \not= 0\}
\end{eqnarray*}
where 
$$
f*u(s):=\int^{+\infty }_{-\infty }f(s-t)u(t)dt .
$$
The following result is a consequence of the Weak Spectral Mapping Theorem that relates the circular spectrum and Beurling spectrum of a uniformly continuous function.
\begin{corollary}\label{lem spec and WSM the}
Let $g\in BUC(\R,\X)$. Then
\begin{equation}
\sigma (g)= \overline{e^{i sp(g)}}.
\end{equation}
\end{corollary}

\subsection{Almost periodic functions}
A subset $E\subset {\R}$ is said to be {\it relatively dense} 
if there exists a number $l>0$ ({\it inclusion length})
such that every interval $[a,a+l]$ contains at least one point of $E$.
Let $f$ be a continuous function on ${\R}$ taking values in a complex Banach space ${\X}$.
$f$ is said to be {\it almost periodic in the sense of Bohr}
if to every $\epsilon >0$ there corresponds a
relatively dense set $T(\epsilon , f)$ ({\it of $\epsilon$-periods })
such that
$$
\sup_{t \in {\R}}\| f(t+\tau )-f(t)\| \le \epsilon  , \ \forall \tau \in T(\epsilon , f).
$$
If $f$ is almost periodic function, then (approximation theorem \cite[Chap. 2]{levzhi}) it can be approximated
uniformly on ${\R}$  by a sequence of trigonometric polynomials, i.e.,
a sequence of functions in $t\in {\R}$ of the form
\begin{equation}\label{trig pol}
P_n(t):=\sum_{k=1}^{N(n)}a_{n,k}e^{i\lambda_{n,k}t}, \ \ n=1,2,...; \lambda_{n,k} \in {\R}, a_{n,k}\in {\X}, t\in {\R}.
\end{equation}
Of course, every function which can be approximated by a sequence of trigonometric polynomials is almost periodic.
Specifically, the exponents of the trigonometric polynomials (i.e., the reals $\lambda_{n,k}$ in (\ref{trig pol})) can be chosen from the set of 
all reals $\lambda$ ({\it Fourier exponents}) such that the following integrals ({\it Fourier coefficients}) 
$$
a(\lambda , f):=\lim_{T\to \infty}\frac{1}{2T}\int^T_{-T}f(t)e^{-i\lambda t}dt
$$
are different from $0$. As is known, there are at most countably such reals $\lambda$, the set of which will
be denoted by $\sigma_b(f)$ and called {\it Bohr spectrum} of $f$. Throughout the paper we will use
the relation $sp(f)= \overline{\sigma_b(f)}$.

\medskip
If $g\in BUC(\R,\X)$ with countable $\sigma (g)$, then its Beurling spectrum $sp(g)$ is also countable by Corollary \ref{lem spec and WSM the}. Therefore, if $\X$ does not contain any space isomorphic to $c_0$ (the space of all numerical sequences converging to zero), the function $g$ is almost periodic (see e.g. \cite{levzhi}). If $\X$ is convex it does not contain $c_0$.

\subsection{Evolutionary processes and the associated evolution semigroups}
\begin{definition}
Let $(U(t,s))_{t\ge s}$ be a two-parameter family of bounded operators in a Banach space $\X$. Then, it is called an evolutionary process if
\begin{enumerate}
\item $U(t,t)=I$\ \ for all \ $t\in {\R}$,
\item $U(t,s)U(s,r)=U(t,r)$\ \ for all \ \ $t\ge s\ge r$,
\item The map \ \ $(t,s)\mapsto U(t,s)x$\ \ is continuous for every fixed \ \ $x
\in {\X}$,
\item $\| U(t,s)\| < Ne^{\omega (t-s)} $ for 
some positive \ $N, \omega $ \ independent of $t \geq s$ .
\end{enumerate}
\end{definition}
An evolutionary process is called {\it 1-periodic} if
\begin{equation*}
U(t+1 ,s+1 )=U(t,s), \ \mbox{for all} \ t \ge s .
\end{equation*}

Recall that for a given 1-periodic evolutionary process
$(U(t,s))_{t\ge s}$ the following operator
\begin{equation*}
M(t):= U(t,t-1), t \in {\R}
\end{equation*}
is called {\it monodromy operator} (or sometime {\it period map,
Poincar\'e map}). Thus we have a family of monodromy operators. We
will denote $M:= M(0)$. The nonzero eigenvalues of $M(t)$  are
called {\it characteristic multipliers}. An important property of
monodromy operators is stated in the following lemma whose proof
can be found in \cite{hen,hinnaiminshi}.
\begin{lemma}\label{lem 5.2}
Under the notation as above the following assertions hold:
\begin{enumerate}
\item $M(t+1) = M(t)$ for all $t$; characteristic multipliers are
independent of time, i.e. the nonzero eigenvalues of $M(t)$ coincide
with those of $M$, \item $\sigma (M(t)) \backslash \{0\}=\sigma (M)
\backslash \{0\}$, i.e., it is independent of $t$, \item If $\lambda
\in \rho (M)$, then the resolvent $R(\lambda , M(t))$ is strongly
continuous, \item If ${\cal M}$ denotes the operator of
multiplication by $M(t)$ in any one of the function spaces $BUC(\R,\X)$ or $AP(\X)$, then
\begin{equation}\label{5.3}
\sigma ({\cal M}) \backslash \{0\} \subset \sigma (M)\backslash
\{0\}.
\end{equation}
\end{enumerate}
\end{lemma}

\medskip
Given an evolutionary process $(U(t,s))_{t\ge s}$, the following semigroup $(T^h)_{h\ge 0}$ is called its associated evolution semigroup
\begin{equation}
T^hg:= U(t,t-h)g(t-h), \ t\in \R , g\in BUC(\R,\X) .
\end{equation}
In general, the evolution semigroup associated with a 1-periodic evolutionary process may not be strongly continuous in the whole space $BUC(\R,\X)$, but in a closed subspace $F$ that includes all elements of $AP(\X)$ and mild solutions in the above sense (see e.g. \cite{bathutrab}, \cite{naimin}). To describe the evolution semigroup associated with a given $(U(t,s))_{t\ge s}$ we consider the following integral equation
\begin{equation}\label{int equ}
u(t)=U(t,s)u(s) +\int^t_s U(t,\xi )f(\xi )d\xi ,\ \mbox{for all} \ \ t \ge s ,
\end{equation}
where $f$ is an element of $BUC(\R,\X)$. We recall the following linear operator  
${\mathcal L}: D({\mathcal L})\subset BUC({\R},{\X}) \to BUC({\R},{\X})$, where $D({\mathcal L})$
consists of all solutions of Eq.(\ref{int equ}) $u(\cdot ) \in BUC({\R},{\X})$  with some $f\in BUC({\R},{\X})$. If $u\in D({\cal L})$, then we define ${\mathcal L}u(\cdot ):= f$.
This operator ${\mathcal L}$ is well defined as a singled-valued operator and is obviously an extension of the differential
operator $d/dt -A$ (see e.g. \cite{murnaimin}). Below, by abuse of notation, we will use the same notation ${\mathcal L}$
to designate its restriction to closed subspaces of $BUC({\R},{\X})$ if this does not make any confusion.

\medskip
If $(T(t))_{t\ge 0}$ is a $C_0$-semigroup in a Banach space $\X$, then $U(t,s):= T(t-s)$  determines a 1-periodic evolutionary process.
\subsection{Mild solutions of Eq.(\ref{FDE}) and a variation of constants formula}
\begin{definition}
A continuous function $u(\cdot )$ on ${\R}$ is said to be a mild solution on ${\R}$ of 
Eq.(\ref{FDE}) with initial $\phi \in C_r$, and is denoted by $u(\cdot ,s, \phi ,f)$
if $u_s =\phi$ and for all $t > s $
\begin{equation}\label{mild solution}
u(t)=T(t-s)\phi (0)+\int^t_sT(t-\xi )[F(\xi )u_\xi +f(\xi )]d\xi .
\end{equation}
\end{definition}
A function $u\in BC(\R,\X)$ is said to be a mild solution of (\ref{FDE}) on $\R$ if
\begin{equation}\label{mild solution2}
u(t)=T(t-s)u(s)+\int^t_sT(t-\xi )[F(\xi )u_\xi +f(\xi )]d\xi , \ \mbox{for all} \ t\ge s .
\end{equation}

Below we will denote by $\mathcal F$ the operator acting on $BUC({\R},{\X})$ defined by the formula
$$
{\mathcal F}u(\xi ):= F(\xi )u_\xi , \ \forall u \in BUC({\R},{\X}).
$$
The following results can be verified directly following the lines in \cite{bathutrab,minrabsch,naimin}.
\begin{lemma}\label{C2S1 lemma 2.1.1}
Let $(T^h)_{h\ge 0}$ be the evolution semigroup associated with a given strongly continuous semigroup
$(T(t))_{t\ge s}$ and 
${\mathcal S}$ denote the space of all elements of $BUC({\R},{\X})$ at which $(T^h)_{h\ge 0}$
is strongly continuous. Then the following assertions hold true:
\begin{enumerate}
\item \ Every mild solution $u\in BUC({\R},{\X})$ of Eq.(\ref{FDE}) is an element of  ${\mathcal S}$, 
\item \ $AP({\X})\subset {\mathcal S}$,
\item \ For the infinitesimal
generator ${\mathcal G}$ of $(T^h)_{h\ge 0}$ in the space ${\mathcal S}$
one has the relation: ${\mathcal  G}g=-{\mathcal L}g$ if $g \in D({\mathcal G})$.
\end{enumerate}
\end{lemma}
For bounded uniformly continuous mild solutions $x(\cdot )$ the following characterization
is very useful:
\begin{theorem}\label{the mild solution}
$x(\cdot )$ is a bounded uniformly continuous mild solution of Eq.(\ref{FDE}) if and only if ${\mathcal L}x(\cdot )={\mathcal F}x(\cdot )+f$.
\end{theorem}

\medskip
As is well known, the homogeneous equation associated with (\ref{FDE}) generates an evolutionary process $(U(t,s))_{t\ge s}$ in the space $C_r =C([-r,0],\X)$. In fact,
\begin{equation}
U(t,s): C_r \ni \phi \mapsto u_t \in C_r ,
\end{equation}
where $u$ is the solution of the equation
\begin{eqnarray*}
u(\tau ) &=& T(\tau -s) \phi (0) +\int^\tau _s T(\tau -\xi )F(\xi )u_\xi d\xi , \ \tau \ge s, \\
u_s &=& \phi .
\end{eqnarray*}
We introduce a function 
$\Gamma^n$ defined by 
$$\Gamma^n(\theta)=\left\{
\begin{array}{cc}
(n\theta+1)I,&\qquad -1/n\leq \theta\leq 0\\
&\\
0,&\theta<-1/n,
\end{array}
\right.
$$
where $n$ is any positive integer and $I$ is the identity operator on ${\mathbb X}$.  
Since the evolutionary process $(U(t,s))_{t\ge s}$ is strongly continuous,
the ${C_r}$-valued function $U(t,s)\Gamma^nf(s)$ is 
continuous in $s\in (-\infty, t]$ whenever $f\in {\rm BC}({\mathbb R},{\mathbb X}) .$ 

\medskip
The following theorem, whose proof could be found in \cite{murnaimin2}, is a variation of constant formula for solutions of (\ref{FDE}) in the phase space $C_r$:
\begin{theorem}\label{the vcf}
The segment $u_t(s,\phi;f)$ of solution $u(\cdot,s,\phi,f)$ of (\ref{FDE}) satisfies the following relation in $C_r$:
\begin{equation}\label{vcf}
u_t(s,\phi;f)=U(t,s)\phi
+\lim_{n\to \infty}\int_{s}^tU(t,\xi )\Gamma^nf(\xi )d\xi , \qquad t\ge s .
\end{equation}
Moreover, the above limit exists uniformly for bounded $|t-s |$.
\end{theorem}

\section{Existence of almost periodic solutions of Eq.(\ref{FDE})}
The result below is an upper estimate of the spectrum of a mild solution to (\ref{FDE}) that  is a key to understand the behavior of a bounded and uniformly continuous mild solution of (\ref{FDE}).
\begin{lemma}\label{lem spectral inclusion}
Let $u$ be a bounded and uniformly continuous mild solution of the equation (\ref{FDE}). Then, the following estimate holds
\begin{equation}\label{spec est}
\sigma (u) \subset \sigma_\Gamma (M) \cup \sigma (f).
\end{equation}
where $\sigma _\Gamma (M):=\{ z\in \C : \ |z| =1, z \in \sigma (M)\}  .$
\end{lemma}
\begin{proof}
By the formula (\ref{vcf})
\begin{eqnarray}\label{est-1}
u_t &=& U(t,t-1) u_{t-1} + \lim_{n\to \infty} \int^t_{t-1} U(t,s)\Gamma^n f(s)ds,
\end{eqnarray}
and the limit exists uniformly for all bounded $t$. First, as $f$ is uniformly continuous and bounded we can see that the function
\begin{equation}
A: \R \ni t \mapsto \lim_{n\to \infty} \int^t_{t-1} U(t,s)\Gamma^n f(s)ds \in C_r
\end{equation}
is also bounded and uniformly continuous. We can check easily the valadity of the identity
$$
\lambda R(\lambda ,S)S(-1) =R(\lambda ,S)+S(-1),
$$
for any $|\lambda | \not = 1$, where $S(t)$ stands for the translation group, and $S:=S(1)$. Note that the operator
${\cal M}$ of multiplication by $M(t)$ commutes with $S$ since the evolutionary process $(U(t,s))_{t\ge s}$ is 1-periodic. Below we will denote by ${\omega}$ the function $\R \ni t \mapsto u_t \in C_r$.
Then, from the identity (\ref{est-1}) one has (for all $\lambda \not= 0$ and $|\lambda |\not =1$)
\begin{eqnarray*}
\lambda R(\lambda,S){\omega} &=& \lambda R(\lambda ,S){\cal M} S(-1) {\omega} + \lambda R(\lambda ,S)A .
\end{eqnarray*}
Therefore,
\begin{eqnarray*}
\lambda R(\lambda,S){\omega} - {\cal M} R(\lambda,S){\omega}&=& {\cal M} S(-1){\omega}  + \lambda R(\lambda ,S)A,\\
(\lambda -{\cal M} )R(\lambda , S)\omega  &=& {\cal M} S(-1){\omega} +\lambda R(\lambda ,S)A .
\end{eqnarray*}
As shown in \cite[Lemma 5.3]{mingasste} for each fixed $n\in \N$
\begin{equation*}
\sigma (G_nf) \subset \sigma (f),
\end{equation*}
where 
$$
G_n f (t):= \int^t_{t-1} U(t,s)\Gamma^n f(s)ds.
$$
As the limit in the formula (\ref{vcf}) is uniform in $t$ we can see that $\sigma (A) \subset \sigma (f)$. Finally, if $\lambda_0 \not\in (\sigma_\Gamma (M) \cup \sigma (f))$, then near $\lambda_0$ the following holds
\begin{eqnarray}
R(\lambda , S)\omega = R(\lambda , {\cal M}) ( {\cal M} S(-1){\omega} +\lambda R(\lambda ,S)A ).
\end{eqnarray}
This shows that the complex function $R(\lambda , S)\omega$ is defined as an analytic function in a neighborhood of $\lambda_0$. 

\medskip
We will show further that this yields that the function  $R(\lambda , S)\omega (0)$ is also defined and analytic in a neighborhood of $\lambda_0$. In fact, before we proceed that we introduce $p:C_r \to \X$ defined as $p(w):= w(0)$. If so, with our above notations $p\circ \omega  = u$, and $p\circ S^k\omega =S^k u$ for all $k\in \N$. If $|\lambda|>1$ we have
\begin{eqnarray*}
p\circ R(\lambda , S) \omega  &=& \lambda^{-1} p\circ (I-S/\lambda )^{-1} \omega\\
&=& \lambda^{-1} p\circ  \left(\sum_{k=0}^\infty S^k/\lambda^k\right)\omega \\
&=&  \lambda^{-1} \left(\sum_{k=0}^\infty S^k/ \lambda^k \right) u \\
&=& R(\lambda ,S) u.
\end{eqnarray*}
Note that for simplicity we make an abuse of notation by denoting also by $S$ the translation in the function space $BUC(\R,\X)$ as well as in $BUC (\R , C_r)$. Similarly, for $\lambda\not= 0$ and $|\lambda | <1$ we can show that
$p\circ R(\lambda , S) \omega =R(\lambda ,S) u$. Hence, the transform $R(\lambda ,S)u$ of the function $u$ has 
$p\circ R(\lambda , S) \omega $
as an analytic extension in a neighborhood of $\lambda _0$. This shows that (\ref{spec est}) holds true, finishing the proof of the lemma.
\end{proof}

\medskip
Next, we recall some concepts and results in \cite{naiminshi}. Note that although the proofs could be found in \cite{naiminshi} we would like to give some new ones that seem to be simpler and would be more convenient to the reader.

Let us consider the subspace \ ${\cal N}\subset BUC({\R},{\X})$ (or $AP(\X)$, respectively)
consisting of all functions \ $v \in BUC({\R},{\X})$ (or $AP(\X)$, respectively)\
such that
\begin{equation}
\sigma (v) \subset S_1 \cup S_2 \ ,
\end{equation}
where $S_1, S_2$ are disjoint closed subsets of the unit
circle $\Gamma$. 
\begin{lemma}\label{lem decomposition}
Under the above notations and assumptions the function space ${\cal N}$
can be split into a direct sum ${\cal N}={\cal N}_1\oplus {\cal N}_2 $
such that $ v \in {\cal N}_i$ if and only if $\sigma (v)\subset S_i$ 
for $i=1,2$. Moreover, any  bounded linear operator in $BUC({\R},{\X})$ (or $AP(\X)$, respectively), that commutes with the translation $S$,
leaves invariant ${\mathcal N}$ as well as ${\mathcal N}_j$, $j=1,2$.
\end{lemma}
\begin{proof}
By Lemma \ref{lem spec of S on Lambda (X)} and the Riezs spectral projection the space ${\cal N}$ could be split into the direct sum ${\cal N}={\cal N}_1\oplus {\cal N}_2 $ with ${\cal N}_1$ is the image of the projection
\begin{equation*}
P:= \frac{1}{ 2 {i\pi}}\int_\gamma R(\lambda , S|_{{\cal N}})d\lambda \ , 
\end{equation*}
where $\gamma $ is a positively oriented contour enclosing $S_1$ and disjoint from $S_2$. We have 
\begin{equation*}
\sigma (S|_{{\cal N}_1}) \subset S_1; \  \sigma (S|_{{\cal N}_2}) \subset S_2.
\end{equation*}
Therefore, if $v\in {\cal N}_i$, ($i=1,2$) by the definition of the circular spectrum it is easy to see that 
\begin{equation*}
\sigma (v) \subset \sigma (S|_{{\cal N}_i}) \subset S_i.
\end{equation*}
The second claim is obvious as any bounded linear operator in $BUC(\R,\X)$ (or $AP(\X)$, respectively) that commutes with $S$ must commute with $P$, so it leaves the spaces ${\cal N}, {\cal N}_1, {\cal N}_2$ invariant.
\end{proof}

\begin{theorem}\label{the main} (Decomposition Theorem) 
Let the following condition be satisfied
\begin{enumerate}
\item Eq.(\ref{FDE}) has a mild solution $u\in BUC(\R,\X)$ (or in $AP(\X)$, respectively)
\item 
\begin{equation}\label{resonnant}
\sigma_\Gamma (M)\ \backslash \sigma (f) \ \mbox{be closed}.
\end{equation}
\end{enumerate}
Then
there exists a mild solution $w$ of Eq.(\ref{FDE}) in $BUC(\R,\X)$ (or $AP(\X)$, respectively) such that
\begin{equation}
\sigma (w)\subset \sigma (f),
\end{equation}
that is unique if
\begin{equation}\label{uniqueness}
\sigma_\Gamma (M) \cap \sigma (f)=\emptyset .
\end{equation} 
\end{theorem}
\begin{proof}
By Lemma \ref{lem spectral inclusion}
\begin{equation}
\sigma (u) \subset \sigma_\Gamma(M)\cup \sigma (f).
\end{equation}
Let us denote by $\Lambda$ the set $\sigma _\Gamma (M) \cup \sigma (f)$, $S_1$ the set
$\sigma (f)$ and $S_2$ the set $\sigma_\Gamma (M)\ \backslash \
\sigma (f)$, respectively. 
Thus, these two sets are closed and disjoint subsets of the unit circle $\Gamma$, so by Lemma \ref{lem decomposition} there exists the projection $P$ from 
${\mathcal N}$ onto ${\mathcal N}_1$ which is commutative with ${\mathcal F}$ and $T^h$. Since $u$ is a mild solution of (\ref{FDE}) if and only if $u\in D({\cal L}) $ and 
\begin{equation}
{\cal L}u = {\cal F}u+f,
\end{equation}
by Lemma \ref{C2S1 lemma 2.1.1}
we have
\begin{eqnarray*}
 {\cal L}u &=& -{\cal G}u,
\end{eqnarray*}
so this yields
\begin{eqnarray*}
P{\cal L}u &=&-P{\cal G}u\nonumber\\
&=& -P\lim_{h\to 0^+}{\frac{T^hu-u}{h}}\nonumber\\
&=&-\lim_{h\to 0^+}P{\frac{T^hu-u}{h}}\nonumber\\
&=&-\lim_{h\to 0^+}{\frac{T^hPu-Pu}{h}}\nonumber\\
&=&  -{\cal G}Pu\nonumber\\
&=&
{\mathcal L}Pu.\label{**}
\end{eqnarray*}
Since $Pf=f$ and $P$ commutes with ${\mathcal F}$,
\begin{eqnarray*}
P{\cal L}u &=& P{\cal F}u+Pf\nonumber\\
{\cal L}Pu &=& {\cal F}Pu + f.
\end{eqnarray*}
By Theorem \ref{the mild solution} this shows $w:=Pu\in {\cal N}_1$ is a mild solution of Eq. (\ref{FDE}) that has circular spectrum $\sigma (Pu) \subset S_1 =\sigma (f)$.
Next, if condition (\ref{uniqueness}) holds, then
the uniqueness of such a solution in ${\cal N}_1$ is clear. In fact, suppose that there is another mild solution $v\in BUC({\R},{\X})$ (or in $AP(\X)$, respectively)
to Eq.(\ref{FDE}) such that 
$\sigma (v) \subset \sigma (f)$, then 
$w-v$ is a mild solution of the homogeneous equation corresponding to Eq.(\ref{FDE}), so
$\sigma (w-v)\subset \sigma_\Gamma (M)$. As $\sigma (v)\subset \sigma (f)$, by (\ref{uniqueness}) this yields that $\sigma (w-v)=\emptyset$, and because of this $w-v=0$. This completes the proof of the theorem.
\end{proof}

Recall that the set of all real numerical sequences that are convergent to zero is a Banach space with sup-norm that is denoted by $c_0$.
As a consequence of the above theorem we obtain the following main result of the paper. 

\begin{theorem}\label{cor main1}
Assume that Eq. (\ref{FDE}) has a bounded uniformly continuous mild solution $u$, and 
Condition (\ref{resonnant}) of Theorem \ref{the main} is satisfied. Moreover, let the space ${\X}$
not contain $c_0$ and $\sigma (f)$ be countable. Then there exists an almost periodic
mild solution $w$ to Eq.(\ref{FDE}) such that $\sigma (w) \subset \sigma (f)$ . Furthermore, if 
(\ref{uniqueness}) holds, then such a solution $w$ is unique.
\end{theorem}
\begin{proof}
The proof is obvious in view of \cite[Theorem 4, p.92]{levzhi} and Theorem \ref{the main}.
\end{proof}

Below we will relax the condition on the existence of a bounded uniformly continuous mild solutions when a condition (\ref{uniqueness}) is satisfied.
\begin{theorem}\label{the main 2}
Under the above notation assume that
\begin{equation}\label{uniqueness2}
\sigma_\Gamma (M) \cap \sigma (f)=\emptyset 
\end{equation} 
holds.
Then there exists a unique almost periodic mild solution $w$ to Eq. (\ref{FDE}) such that $\sigma (w) \subset \sigma (f)$.
\end{theorem}
\begin{proof}
Consider the difference equation
\begin{equation}\label{difference eq}
w(t)= M(t) w(t-1)+g(t), \ t\in \R ,
\end{equation}
where for all $t\in \R$
\begin{eqnarray*}
M(t) &:=& U(t,t-1),\\
g(t) &:=&\lim_{n\to \infty} \int^t_{t-1} U(t,s)\Gamma^n f(s)ds.
\end{eqnarray*}
First, we note that $g$ is almost periodic function taking values in $C_r$. In fact, for each $n\in \N$ the function
$$
F_n: \R \ni t \mapsto \Gamma^n f(t) \in C_r
$$
is an almost periodic function with $\sigma (F_n) \subset \sigma (f)$. Next, by \cite[Lemma 5.3]{mingasste}
the function 
$$
F:\R\ni t \mapsto \int^t_{t-1} U(t,\xi ) F_n(\xi )d\xi
$$
is also almost periodic, and $\sigma (F) \subset \sigma (F_n) \subset \sigma (f)$. Therefore, $g$ is almost periodic and $\sigma (g) \subset \sigma (f)$.

\medskip
By \cite[Theorem 4.7]{mingasste} if (\ref{uniqueness2}) holds there exists a unique almost periodic solution $w$ to (\ref{difference eq}) such that $\sigma (w) \subset \sigma (f)$. Our next goal is to prove that there exists a mild solution $u$ of Eq. (\ref{FDE}) such that $u_n=w(n)$ for all $n\in \Z$. For each fixed $n\in \Z$ consider the unique mild solution to Eq. (\ref{FDE}) on the interval $[n,n+1]$ that is generated by the equation
\begin{eqnarray*}
u(t) &=& T(t-n)[w(n)](0) + \int^t_{n} T(t-\eta )[F(\eta )u_\eta +f(\eta )] d\eta , \ t\in [n,n+1],\\
u_n &=& w(n).
\end{eqnarray*}
This solution exists uniquely on the interval $[n,n+1]$ for each $n\in \Z$. By the Variation-of-Constants formula
(\ref{vcf})
\begin{equation}\label{3.29}
u_t=U(t,n)w(n)
+\lim_{m\to \infty}\int_{n}^tU(t,s)\Gamma^mf(s)ds, \qquad t\geq n .
\end{equation}
Therefore, if $t=n+1$ we have that $u_{n+1}= w(n+1)$. This means that we obtain a mild solution $u$ of Eq. (\ref{FDE}) that is defined on each interval $[n,n+1]$ by (\ref{3.29}) so that it coincides with $w$ at each integer $n$. Therefore, the sequence $w(n) = u_n$ is almost periodic. This yields that $u(n) = u_n(0)$ is an almost periodic sequence. We are going to prove that $u$ is almost periodic function. The proof will follow a well known idea in \cite{fin} that are used in \cite{bathutrab,furnaimin} as well. For the completeness we present it below.

As $w(\cdot )$ and $f$ are almost periodic, so is the function $g:{\bf R}\ni t \mapsto (w(t),f(t))\in C\times {\bf X}$
(see \cite[p.6]{levzhi}). 
As is known, the sequence
$\{g(n)\}=\{(w(n),f(n))\}$ is almost periodic.
Hence, for every positive $\epsilon$ the following set is relatively dense 
(see \cite[p. 163-164]{fin})
\begin{equation*}
T:={\bf Z} \cap T(g, \epsilon ),
\end{equation*}
where $T(g,\epsilon ):= \{ \tau \in {\bf R}: 
\sup_{t\in {\bf R}}\| g(t+\tau )-g(t)\| < \epsilon \}$, 
i.e., the set of $\epsilon$
periods of $g$. Hence, for every $m \in T$ we have
\begin{eqnarray*}
\|f(t+m)-f(t)\| &<& \epsilon , \forall t \in {\bf R}, \\ 
\| w(n+m)-w(n)\| &<& \epsilon , \forall n \in {\bf Z}.  
\end{eqnarray*}
Since $u$ is a solution to Eq.(\ref{mild solution}), for $0\le s < 1$ and all $n \in {\bf N}$, we have
\begin{eqnarray*}
&&\| u(n+m+s)-u(n+s)\| \le\| T(s)\| \cdot \| w(n+m)-w(n)\| \\
&&\hspace{2cm} +\int^s_0\| T(s-\xi )\| \big[ \sup_{t}\| F(t) \| \cdot \| u_{n+m+\xi }-u_{n+\xi }\| \\
&&\hspace{2cm} + \| f(n+m+\xi )-f(n+\xi )\| \big]d\xi \\
&&\ \ \le Ne^\omega \| w(n+m)-w(n)\| +Ne^\omega  \int^s_0 \big[ \| F \| \\
&&\ \hspace{1cm} \times \| u_{n+m+\xi }-u_{n+\xi }\| + \| f(n+m+\xi )-f(n+\xi )\| \big] d\xi .
\end{eqnarray*}
Hence
\begin{eqnarray*}
&&\| x_{n+m+s}-x_{n+s}\|
\le  Ne^\omega \| w(n+m)-w(n)\| \\
&&\ \ \ \ +Ne^\omega  \int^s_0 \big[ \| F \| \cdot \| x_{n+m+\xi }-x_{n+\xi }\| + \| f(n+m+\xi )-f(n+\xi )\| \big]d\xi .
\end{eqnarray*}
Using the Gronwall inequality we can show that
\begin{equation}
\| u_{n+m+s}-u_{n+s}\| \le 
\epsilon M,
\end{equation}
where $M$ is a constant which depends only on $\sup_{t}\| F(t)\| , N, \omega $. This shows that 
$m$ is a $\epsilon M$-period of 
the function $x(\cdot )$. 
Finally, since $T$ is relatively dense for every $\epsilon$,
we see that $x(\cdot )$ is an almost periodic mild solution of Eq.(\ref{FDE}).
Once the almost periodicity of $u$ was proved we are are able to apply the Decomposition Theorem \ref{the main} to finish the proof of this theorem.
\end{proof}

\section{Discussion: Variation-of-constant formula in the phase spec and further extension}
Our results in the previous section could be extended to a bit more general case of periodic equations. Namely, let us consider equations of the form
\begin{equation}\label{per fde}
\frac{du}{dt} = A(t)u+ F(t)u_t + f(t), \ t\in \R ,
\end{equation}
where the family of (possibly unbounded) operators $A(t)$ generates a $1$-periodic evolutionary process and $F(t )$ is a 1-periodic family of bounded operators as in (\ref{FDE}), and $f$ is an almost periodic function taking values in $\X$.

\medskip
The presentation of our proofs of the results in the previous section relies on the variation-of-constants formula (\ref{vcf}) in the phase space $C_r$ that allows us to easily outline the ideas. In turn, we have made use of the formula available in the case when $A(t)$ is independent of $t$ although our results could be true even if $A(t)$ may depend on $t$ periodically with the same period as that of $F(t)$.

\medskip
As shown in \cite[Lemma 4.1]{furnaimin}, there is a way to get around with the variation-of-constant formula (\ref{vcf}). Below is a version of Lemma 4.1 from \cite{furnaimin} that could be used to extend our results in the previous section to the general case of equations (\ref{per fde}). We consider the following Cauchy Problem
for each given $t\in \R$\begin{eqnarray*}
y(\xi ) &=& \int^\xi _{t-1} V(\xi ,\eta )[F(\eta ) y_\eta +f(\eta )]d\eta , \ \xi \ge t-1,\\
y_{t-1}&=& 0 \in C_r ,
\end{eqnarray*}
where $(V(t,s))_{t\ge s}$ is a 1-periodic evolutionary process generated by the homogeneous equation
$$
\frac{du}{dt}= A(t)u,
$$
Let us define $v:\R \ni t \mapsto y_t \in C_r$. We define the operator $L: BUC(\R,\X) \ni f \mapsto v$.
\begin{lemma}
The operator $L$ is well defined operator in $BUC(\R,\X)$ that is linear and continuous and commutes with the translation $S$.
\end{lemma}
\begin{proof}
Since the proof could be easily adapted from that of \cite[Lemma 4.1]{furnaimin} details will be omitted.
\end{proof}
From the definition of the function $v$ we can verify that if $u$ is a mild solution of (\ref{FDE}) on the real line, then
\begin{equation*}
u_t=U(t,t-1)u_{t-1} + v (t), \ t\in\R .
\end{equation*}
Therefore, the circular spectrum of $u$ could be estimated as below
\begin{lemma}
\begin{equation*}
\sigma (u) \subset \sigma_\Gamma (M) \cup \sigma (f) .
\end{equation*}
\end{lemma}
\begin{proof}
Since $L$ is linear, bounded and commutes with $S$ we have  $\sigma (v) =\sigma (Lf) \subset \sigma (f)$. The rest of the proof is similar to that of Lemma \ref{spec est}.
\end{proof}
All main results of the previous section, Theorems \ref{the main}, \ref{cor main1} and  \ref{the main 2}
will follow if we adjust the technique of decomposition as discussed in \cite{naiminshi} to periodic evolutionary processes. 

\bibliographystyle{amsplain}

\end{document}